\documentclass{amsart}
\usepackage{amsmath}
\usepackage{amsfonts}
\usepackage{enumerate}
\usepackage{amssymb}
\usepackage{amscd}
\usepackage[all]{xy}
\usepackage{bm}

 \DeclareMathOperator{\SL}{SL}
 
\DeclareMathOperator{\Ext}{Ext} \DeclareMathOperator{\Ann}{Ann}
 \DeclareMathOperator{\Hom}{Hom}
\DeclareMathOperator{\GL}{GL} \DeclareMathOperator{\Gal}{Gal}

\begin{document}
\theoremstyle{plain}
\newtheorem{MainThm}{Theorem}
\renewcommand{\theMainThm}{\Alph{MainThm}}
\newtheorem{MainCor}{Corollary}
\renewcommand{\theMainCor}{\Alph{MainCor}}
\newtheorem*{trm}{Theorem}
\newtheorem*{lem}{Lemma}
\newtheorem*{prop}{Proposition}
\newtheorem*{defn}{Definition}
\newtheorem*{thm}{Theorem}
\newtheorem*{example}{Example}
\newtheorem*{cor}{Corollary}
\newtheorem*{conj}{Conjecture}
\newtheorem*{hyp}{Hypothesis}
\newtheorem*{thrm}{Theorem}
\newtheorem*{quest}{Question}
\theoremstyle{remark}
\newtheorem*{rem}{Remark}
\newtheorem*{rems}{Remarks}
\newtheorem*{notn}{Notation}
\newcommand{\Fp}{\mathbb{F}_p}
\newcommand{\Fq}{\mathbb{F}_q}
\newcommand{\Zp}{\mathbb{Z}_p}
\newcommand{\Qp}{\mathbb{Q}_p}
\newcommand{\Kr}{\mathcal{K}}
\newcommand{\Rees}[1]{\widetilde{#1}}
\newcommand{\invlim}{\lim\limits_{\longleftarrow}}
\newcommand{\Md}[1]{\mathcal{M}(#1)}
\newcommand{\Pj}[1]{\mathcal{P}(#1)}
\newcommand{\G}[2]{\mathcal{G}_{#1}(#2)}

\title{Centres of skewfields and completely faithful Iwasawa modules}
\author{Konstantin Ardakov}

\begin{abstract} Let $H$ be a torsionfree compact $p$-adic analytic group whose Lie algebra is split semisimple. We show that the quotient skewfield of fractions of the Iwasawa algebra $\Lambda_H$ of $H$ has trivial centre and use this result to classify the prime $c$-ideals in the Iwasawa algebra $\Lambda_G$ of $G := H \times \Zp$. We also show that a finitely generated torsion $\Lambda_G$-module having no non-zero pseudo-null submodule is completely faithful if and only if it is has no central torsion. This has an application to the study of Selmer groups of elliptic curves.
\end{abstract}

\maketitle
\let\le=\leqslant  \let\leq=\leqslant
\let\ge=\geqslant  \let\geq=\geqslant

\section{Introduction}

\subsection{Iwasawa algebras}
Let $p$ be a prime number. In this note we are concerned with modules over the Iwasawa algebra
\[\Lambda_G:=\lim\limits_{\longleftarrow}\!_{N\triangleleft_oG} \Zp[G/N]\]
of a compact $p$-adic analytic group $G$. These groups frequently occur as images of Galois representations on $p$-power division points of abelian varieties and act on various arithmetic objects of interest such as ideal class groups and Selmer groups. These arithmetic objects then naturally become modules over the associated Iwasawa algebra. By examining the structure of these modules in detail, it is sometimes possible to obtain precise arithmetic information.

\subsection{Selmer groups} As a concrete example of the philosophy outlined above in action, consider the following situation. Let $E$ be an elliptic curve over a number field $F$ and assume $E$ has no complex multiplication. Let $E_{p^\infty}$ denote the group of all $p$-power division points on $E$, let $F_{\infty} = F(E_{p^\infty})$ and let $G = \Gal(F_{\infty} / F)$. By a well-known theorem of Serre \cite{Serre}, $G$ is an open subgroup of $\GL_2(\Zp)$. In this setting, Coates, Schneider and Sujatha define the \emph{Selmer group} $\mathcal{S}(E/F_\infty)$ \cite[Definition 8.1]{CSS} and study the Pontryagin dual $X(E/F_{\infty})$ of $\mathcal{S}(E/F_\infty)$. They show that $X(E/F_\infty)$ is always a finitely generated $\Lambda_G$-module. Moreover, they show that under suitable assumptions on $E$ and $p$, $X(E/F_\infty)$ is in fact finitely generated and torsionfree over $\Lambda_H$ where $H$ is a certain closed normal subgroup of $G$ satisfying $G/H \cong \Zp$.    They also posed a number of questions concerning the structure of $X(E/F_\infty)$ as a $\Lambda_G$-module, including the following:
\begin{itemize}
\item Let $Z$ denote the centre of $G$. Is $X(E/F_\infty)$ torsionfree as a $\Lambda_Z$-module?
\item Is $q(X(E/F_\infty))$ completely faithful?
\end{itemize}

\subsection{Completely faithful modules}\label{Main} Let $R$ be a
Noetherian domain, which is a maximal order in its skewfield of
fractions. Let $\mathcal{M}$ denote the category of all finitely
generated $R$-modules, let $\mathcal{C}$ denote the full subcategory
of all pseudo-null modules in $\mathcal{M}$ and let $q : \mathcal{M}
\to \mathcal{M} / \mathcal{C}$ denote the quotient functor. Recall
\cite[\S 5]{CSS} that the \emph{annihilator ideal} of an object $A
\in \mathcal{M}/\mathcal{C}$ is defined by
\[\Ann(A) := \sum\{\Ann_R(M) : q(M) \cong A\},\]
and that $A$ is said to be \emph{completely faithful} if $\Ann(B)=0$
for any non-zero subquotient $B$ of $A$. The purpose of this note is
to prove the following
\begin{thm} Let $p \geq 5$, let $H$ be a torsionfree compact $p$-adic analytic group whose Lie algebra $\mathcal{L}(H)$ is split semisimple over $\Qp$ and let $G = H\times Z$ where $Z \cong \Zp$. Then $\Lambda_G$ is a maximal order. Let $M$ be a finitely generated torsion $\Lambda_G$-module, which has no non-zero pseudo-null submodules. Then $q(M)$ is completely faithful if and only if $M$ is torsionfree over $\Lambda_Z$.
\end{thm}
The proof is given in $\S$\ref{Proof}. We should point out that if $M$ is a finitely generated $\Lambda_G$-module which is also finitely generated over $\Lambda_H$, then $M$ is automatically $\Lambda_G$-torsion. Moreover, in this situation, $M$ has no nonzero pseudo-null submodules if and only if $M$ is torsionfree over $\Lambda_H$.

Theorem \ref{Main} should be compared with \cite[Theorem 6.3]{Ven}, which states that if instead $G$ is a semidirect product of two copies of $\Zp$, then $q(M)$ is completely faithful whenever $M$ is a finitely generated $\Lambda_G$-module which is finitely generated over $\Lambda_H$.

The assumptions on the group $G$ in Theorem \ref{Main} are fairly mild. For example, any open pro-$p$ subgroup $G$ of $\GL_2(\Zp)$ satisfies them whenever $p\geq 3$: just take $H = G \cap \SL_2(\Zp)$. Moreover, if $G$ is just assumed to have a closed normal subgroup $H$ such that $\mathcal{L}(H)$ is semisimple and such that $G/H \cong \Zp$, then by passing to an open subgroup we can ensure that in fact $G = H \times \Zp$. This is because the only extension of a semisimple Lie algebra by a one-dimensional Lie algebra is the trivial extension.

\subsection{Selmer groups revisited} Returning to our $\GL_2$ example, Theorem \ref{Main} implies that $q(X(E/F_\infty))$ must necessarily be completely faithful, provided $X(E/F_\infty)$ is finitely generated over $\Lambda_H$ and torsionfree over both $\Lambda_H$ and $\Lambda_Z$. Although $X(E/F_\infty)$ is known to be finitely generated over $\Lambda_H$ in many circumstances, there are at present no known examples where one can prove that $X(E/F_\infty)$ is torsionfree over $\Lambda_Z$.

We should remark at this point that Coates \emph{et al} conjecture that $Y(E/F_\infty) := X(E/F_\infty) / X(E/F_\infty)(p)$ is finitely generated over $\Lambda_H$ whenever $p \geq 5$ and $E$ has good ordinary reduction at $p$ \cite[Conjecture 5.1]{CFKSV}. Here $X(E/F_\infty)(p)$ denotes the $p$-torsion submodule of $X(E/F_\infty)$.

\subsection{Centres of skewfields and prime c-ideals} In order to prove Theorem \ref{Main}, we compute the centre $Z(Q(\Lambda_H))$ of the quotient skewfield of fractions $Q(\Lambda_H)$ of $\Lambda_H$, thereby answering a question of Venjakob
 \cite[Question 6.4]{Ven}. We also classify \emph{all} the prime c-ideals in both $\Lambda_H$ and $\Lambda_G$:

\begin{thm} Under the assumptions of Theorem \ref{Main}, the centre of $Q(\Lambda_H)$ is equal to $\Qp$, the only prime c-ideal of $\Lambda_H$ is $p\Lambda_H$ and every prime c-ideal of $\Lambda_G$ apart from $p\Lambda_G$ is generated by a distinguished polynomial in $\Lambda_Z$.
\end{thm}

The proof is given in $\S \ref{CentQ}$ and $\S \ref{PrimesG}$ and follows from a very recent result of F. Wei, J. J. Zhang and the author \cite[Theorem 7.3]{AWZ}, which essentially states that $\Omega_H = \Lambda_H/p\Lambda_H$ has no nontrivial prime c-ideals.

\subsection{Acknowledgements} The author would like to thank John Coates for his very valuable comments.

\section{Preliminaries}
\subsection{Fractional ideals}\label{Frac} Let $R$ be a Noetherian domain. It is well-known that $R$ has a skewfield of fractions $Q$. Recall that a right $R$-submodule $I$ of $Q$ is said to be a \emph{fractional right ideal} if $I$ is non-zero and $I \subseteq uR$ for some non-zero $u\in Q$. Fractional left ideals are defined similarly. If $I$ is a fractional right ideal, then
\[ I^{-1} := \{q\in Q : qI \subseteq R\}\]
is a fractional left ideal and there is a similar definition of $I^{-1}$ for fractional left ideals $I$. There is a natural isomorphism between $I^{-1}$ and $I^\ast = \Hom_R(I,R)$. The following elementary result will be very useful for computing $I^{-1}$.
\begin{lem} Let $R$ be a Noetherian domain and let $I$ be a non-zero right ideal of $R$. Then $I^{-1}/R \cong \Ext_R^1(R/I,R)$ as left $R$-modules.
\end{lem}
Recall that the fractional right ideal $I$ is said to be \emph{reflexive} if
\[I = \overline{I} := (I^{-1})^{-1}.\]
Equivalently, the canonical map $I \to I^{\ast\ast}$ is an isomorphism. Let $I$ be a fractional right ideal which is also a fractional left ideal. We say that $I$ is a \emph{fractional c-ideal} if it is reflexive on both sides, and we say that $I$ is a \emph{c-ideal} if $I$ is contained in $R$. Of particular interest are the c-ideals which happen to be prime ideals of $R$.

\subsection{Two results about prime c-ideals} \label{ModX}
Let $R$ be a Noetherian domain and let $x \in R$ is a non-zero central element such that $R/xR$ is a domain. The following facts are well-known, but we include the proofs for the convenience of the reader because we will use these results repeatedly.

\begin{lem} Suppose that $I$ is a proper c-ideal of $R$ containing $x$. Then $I=xR$.
\end{lem}
\begin{proof} Since $I^{-1}I \subseteq R$ and $I \supseteq xR$, we see that $I^{-1}x \subseteq R$. Because $x$ is central, $I^{-1}x = xI^{-1}$ is a two-sided ideal of $R$. Now $(xI^{-1})I \subseteq xR$ and $xR$ is a prime ideal by assumption, so either $xI^{-1} \subseteq xR$ or $I \subseteq xR$. Since $I$ is proper and reflexive, $R$ is properly contained in $I^{-1}$ and therefore $I = xR$.
\end{proof}

\begin{prop} Suppose that $x$ lies in the Jacobson radical of $R$ and that every non-zero two-sided ideal $J$ of $T:=R/xR$ satisfies $J^{-1} = T$. Then $xR$ is the only prime c-ideal of $R$.
\end{prop}
\begin{proof} Let $I$ be a prime c-ideal of $R$. By the above lemma, it will be enough to show that $x \in I$. Let $M = R/I$ and let $J$ be the image of $I$ in $T$, so that $M/xM \cong T/J$ as right $T$-modules. The short exact sequence of right $R$-modules
\[0 \longrightarrow M \stackrel{x}{\longrightarrow} M \longrightarrow T/J \longrightarrow 0\]
gives rise to the following long exact sequence of left $R$-modules:
\[\cdots \longrightarrow \Ext_R^1(M,R) \stackrel{x}{\longrightarrow} \Ext_R^1(M,R) \longrightarrow \Ext_R^2(T/J, R) \longrightarrow \cdots.\]
Suppose for a contradiction that $x \notin I$. Because $I$ is prime and $x$ is central, the module $M$ is $x$-torsion-free. By the well-known Rees Lemma (see, for example, \cite[Lemma 1.1]{ASZ}), we have
\[\Ext_R^2(T/J, R) \cong \Ext_T^1(T/J,T).\]
Because $J$ is a non-zero two-sided ideal of $T$, we must have $J^{-1} = T$ by assumption. Since $T$ is a domain, $\Ext_T^1(T/J,T) \cong J^{-1}/T = 0$ by Lemma \ref{Frac}, so the finitely generated left $R$-module $E := \Ext_R^1(M,R)$ satisfies $E = xE$. Because $x$ lies in the Jacobson radical of $R$, Nakayama's Lemma implies that $E = 0$. Hence $I^{-1} = R$, again by Lemma \ref{Frac}. But $I$ is reflexive by assumption, so $I = (I^{-1})^{-1} = R$, contradicting the assumption that $I$ is a prime ideal.
\end{proof}

\subsection{Maximal orders}
\label{Asano} Recall that $R$ is said to be a \emph{maximal order} in $Q$ if whenever $S$ is a subring of $Q$ containing $R$ such that $uSv\subseteq R$ for some non-zero $u,v \in Q$, we must have $S = R$. Asano \cite{Asano} showed that the set $G(R)$ of fractional c-ideals for a maximal order $R$ is an abelian group when equipped with the product
\[I \cdot J := \overline{IJ}\]
and inverse $I \mapsto I^{-1}$ defined in $\S$\ref{Frac}. He also proved the following
\begin{thm} $G(R)$ is a free abelian group on the set of prime c-ideals of $R$.
\end{thm}
\begin{proof} See, for example, \cite[II.1.8 and II.2.6]{MR}.
\end{proof}

\section{A useful isomorphism}

\subsection{Venjakob's localisation of Iwasawa algebras}

If $G$ is a compact $p$-adic analytic group, we write $\Lambda_G$ and $\Omega_G$ for the completed group algebras of $G$ with coefficients in $\Zp$ and $\Fp$, respectively. See \cite{AB2} for more information about these algebras.

Let $H$ be a closed normal subgroup of $G$. Recall from \cite[\S 1.5]{AB1} the two-sided augmentation ideal
\[w_{H,G} := (H-1)\Omega_G = \ker(\Omega_G \to \Omega_{G/H}).\]
Let $N$ be an open pro-$p$ subgroup of $H$ which is normal in $G$, and let $P_H = \sqrt{w_{N,G}}$ be the prime radical of $w_{N,G}$. This is a semiprime ideal of $\Omega_G$ which only depends on $H$ by \cite[Lemma 3.2]{AB1}. It was shown in \cite[Theorem D]{AB1} that $P_H$ is a \emph{localisable ideal} in $\Omega_G$, meaning that
\[\{s \in \Omega_G : s\mbox{ is regular modulo $P_H$}\}\]
is a two-sided Ore set in $\Omega_G$. We will denote the Ore localisation of $\Omega_G$ at this Ore set by $\Omega_{G,H}$.

\subsection{Lemma}\label{IsoHfinite} Suppose that $G = Z\times H$ where $H$ is finite and $Z \cong \Zp^n$ for some $n \geq 0$. Then $\Omega_{G,H} \cong K[H]$, where $K$ is the field of fractions of $\Omega_Z$.
\begin{proof} Suppose $L$ is an open subgroup of $H$ which is normal in $G$. Then the definition of $P_H$ shows that $P_H = P_L$, so $\Omega_{G,H} = \Omega_{G,L}$. Because $H$ is finite, we may choose $L$ to be the trivial group. Now $Z$ is an open normal uniform subgroup of $G$, and in this situation \cite[Lemma 5.1]{AB1} shows that $\Omega_{G,H}$ is a crossed product of $\Omega_{Z,1}$ with the finite group $G/Z$:
\[\Omega_{G,H} = \Omega_{G,1} = \Omega_{Z,1} \ast (G/Z).\]
Clearly $\Omega_{Z,1}$ is the field of fractions $K$ of $\Omega_Z$, and $\Omega_G = \Omega_Z[H]$ because $G = Z\times H$. Now inspecting the proof of \cite[Lemma 5.1]{AB1} shows that the crossed product on the right hand side is in fact just the ordinary group ring $K[H]$.
\end{proof}

\subsection{A completion of $\Omega_{G,H}$}\label{Iso} Now let $H$ be an arbitrary compact $p$-adic analytic group and let $Z \cong \Zp^n$. Letting $K$ denote the field of fractions of $\Omega_Z$, we can form the completed group algebra of $H$ with coefficients in $K$:
\[KH := K[[H]] : =\lim_{\longleftarrow} K[H/U],\]
where $U$ runs over all the open normal subgroups of $U$. We can now state the main result of this section.

\begin{thm} Let $G = Z\times H$, let $N$ be an open normal uniform subgroup of $H$ and let $\mathfrak{m}$ be the augmentation ideal $w_{N,G}\Omega_{G,H}$ of $\Omega_{G,H}$. Then
\begin{enumerate}[{(}a{)}]
\item the $\mathfrak{m}$-adic filtration on $\Omega_{G,H}$ is Zariskian, and
\item the completion of $\Omega_{G,H}$ with respect to this filtration is isomorphic to $KH$.
\end{enumerate}
\end{thm}
\begin{proof}(a) Because $N$ is normal in $G$ and has finite index in $H$, $\Omega_{G,H} = \Omega_{G,N}$ as above. Now $ZN$ is an open normal uniform subgroup of $G$ and $G/ZN \cong H/N$, so \cite[Lemma 5.1]{AB1} shows that
\[\Omega_{G,H} = \Omega_{G,N} \cong \Omega_{ZN,N} \ast (H/N).\]
Let $\mathfrak{n} = w_{N,ZN}\Omega_{ZN,N}$, so that $\mathfrak{m} =
\mathfrak{n}\cdot \Omega_{G,N}$ and $\mathfrak{m}^k = \mathfrak{n}^k
\cdot \Omega_{G,N}$ for all $k$. We can therefore compute the Rees
ring of $\Omega_{G,N}$ with respect to the $\mathfrak{m}$-adic
filtration as follows:
\[\widetilde{\Omega_{G,N}} = \bigoplus_k \mathfrak{m}^{-k}t^k = \left(\bigoplus_k \mathfrak{n}^{-k}t^k\right) \otimes_{\Omega_{ZN,N}}\Omega_{G,N} \cong \widetilde{\Omega_{ZN,N}} \ast (H/N).\]
It was shown in \cite[Proposition 5.3]{AB1} that the Rees ring $\widetilde{\Omega_{ZN,N}}$ is Noetherian, so $\widetilde{\Omega_{G,N}}$ is also Noetherian since $H/N$ is finite.

Now $P_H\Omega_{G,H}$ is the Jacobson radical of $\Omega_{G,H}$ by
\cite[Theorem 3.2.3(a)]{Jat}, and $\mathfrak{m} =
w_{N,G}\Omega_{G,H} \subseteq P_H\Omega_{G,H}$. Hence the
$\mathfrak{m}$-adic filtration is Zariskian.

(b) Let $R$ denote the completion of $\Omega_{G,N}$ at the $\mathfrak{m}$-adic filtration. By \cite[Lemma 7.1]{DDMS}, the families of ideals $\{w_{N,N}^k : k\geq 0\}$ and $\{w_{U,N} : U \triangleleft_o N\}$ are cofinal in $\Omega_N$. It follows that the families of ideals
\[\{\mathfrak{m}^k : k\geq 0\} \quad\mbox{and}\quad \{w_{U,G}\Omega_{G,N} : U \triangleleft_o H, U \leq N\}\]
are cofinal in $\Omega_{G,N}$, so $R$ is isomorphic to the completion of $\Omega_{G,N}$ at the second family.

Let $U$ be an open normal subgroup of $H$ which is contained in $N$. Then $G/U \cong Z\times (H/U)$ and $H/U$ is finite. Now $\Omega_G / w_{U,G} \cong \Omega_{G/U}$, so by Lemma \ref{IsoHfinite} we have
\[\Omega_{G,N} / w_{U,G}\Omega_{G,N} \cong \Omega_{G/U, N/U} = \Omega_{G/U, H/U} \cong K[H/U].\]
Hence $R \cong \invlim K[H/U] \cong KH$.
\end{proof}

\subsection{An application}\label{Appl} Let $I_H$ denote the inverse image of $P_H$ in $\Lambda_G$. By \cite[Theorem G]{AB1}, $I_H$ is a localisable ideal in $\Lambda_G$ and we denote the corresponding localisation by $\Lambda_{G,H}$. The element $p$ lies in the Jacobson radical of $\Lambda_{G,H}$ and
\[\Lambda_{G,H} / p\Lambda_{G,H} \cong \Omega_{G,H}.\]

\begin{prop} Let $G = Z \times H$ where $Z \cong \Zp^n$ for some $n$ and $H$ is a torsionfree compact $p$-adic analytic group. Let $K$ be the field of fractions of $\Omega_Z$ and suppose that $J^{-1} = KH$ for every non-zero two-sided ideal $J$ of $KH$. Then the only prime c-ideal of $\Lambda_{G,H}$ is $p\Lambda_{G,H}$.
\end{prop}
\begin{proof} Because $H$ is torsionfree, $KH$ is a Noetherian domain by the proof of \cite[Theorem C]{AB1}. In view of Proposition \ref{ModX}, it will be enough to show that every non-zero two-sided ideal $J$ of $T:=\Omega_{G,H}$ satisfies $J^{-1} = T$.

Fix an open normal uniform subgroup $N$ of $H$ as in Theorem \ref{Iso}. By that result, the $\mathfrak{m}$-adic filtration on $T$ is Zariskian and the completion of $T$ with respect to this filtration is isomorphic to $KH$. Hence $KH$ is a faithfully flat (right and left) $T$-module by \cite[Chapter II, \S 2.2.1, Theorem 2]{LVO}, and therefore
\[KH \otimes_T \Ext_T^1(T/J,T) \cong \Ext_{KH}^1((T/J)\otimes_{T} KH, KH)\]
by \cite[Proposition 1.2]{AWZ}. Now $(T/J)\otimes_{T} KH \cong KH /
J\cdot KH$ and $J \cdot KH$ is just the completion $\widehat{J}$ of
$J$ with respect to the $\mathfrak{m}$-adic topology on $T$ by
\cite[Chapter II, \S 1.1.2, Theorem 10(5)]{LVO}. As such, $J\cdot
KH$ is a \emph{two-sided} ideal of $KH$ and therefore $(J\cdot
KH)^{-1} = KH$ by our assumption.

It follows from Lemma \ref{Frac} that $KH \otimes_T \Ext_T^1(T/J,T) = 0$, but $KH$ is a faithfully flat left $T$-module, so $\Ext_T^1(T/J,T) = 0$. Therefore $J^{-1} = T$.
\end{proof}
\section{The centre of the skewfield of fractions of $\Lambda_H$}

\subsection{Iwasawa algebras are maximal orders}
\label{IwaMaxOrd}
\begin{lem} Let $G$ be a torsionfree compact $p$-adic analytic group. Then $\Lambda_G$ is a maximal order.
\end{lem}
\begin{proof} It is well known that $\Lambda_G$ is a domain. Moreover, $\Omega_G$ is also a domain \cite[Theorem 4.3]{AB2} and is a maximal order in its skewfield of fractions by \cite[Corollary 4.7]{AB2}. If we filter $\Lambda_G$ $p$-adically, we get a Zariskian filtration whose associated graded ring is isomorphic to the polynomial ring $\Omega_G[t]$. This is a maximal order by \cite[V.2.5]{MR}. Because $\Omega_G[t]$ is a domain we may apply \cite[X.2.1]{MR} and deduce that $\Lambda_G$ is also a maximal order.
\end{proof}

If one is happy to pass to open subgroups of $G$ then one can assume that the group $G$ is $p$-valued in the sense of Lazard. In this case, $\Lambda_G$ is already known to be a maximal order \cite[Proposition 7.2]{CSS}.

\subsection{Two-sided ideals in $KH$}\label{Omega}
From now on, we will assume that our group $G$ satisfies the conditions of Theorem \ref{Main}:
\begin{itemize}
\item $G = H \times Z$ for some closed subgroups $H$ and $Z$,
\item $Z \cong \Zp$,
\item $H$ is torsionfree pro-$p$,
\item $\mathcal{L}(H)$ is split semisimple over $\Qp$,
\item $p \geq 5$.
\end{itemize}
We should remark that the hypotheses made on $H$ are only there to ensure that the conclusion of following theorem holds. In particular, the hypothesis that $p\geq 5$ can probably be dispensed with altogether.

\begin{thm} For any field $K$ of characteristic $p$, if $I$ is a non-zero two-sided ideal of the completed group algebra $KH$, then $I^{-1} = KH$.
\end{thm}
\begin{proof} It is shown in \cite[Theorem 7.3]{AWZ} that $KH$ has no non-trivial reflexive two-sided ideals. As $I$ is non-zero, it follows that $\overline{I} = KH$. Hence $I^{-1} = \overline{I}^{-1} = KH$.
\end{proof}

\subsection{}\label{FgOverH} Here is our first application of Theorem \ref{Omega}.

\begin{thm} Let $I$ be a prime c-ideal of $\Lambda_G$. Then either $I = p\Lambda_G$ or $\Lambda_G / I$ is finitely generated over $\Lambda_H$.
\end{thm}
\begin{proof} Let $S$ be the Ore set in $\Lambda_G$ that gives rise to the localisation $\Lambda_{G,H}$:
\[S = \{x \in \Lambda_G : x \mbox{ is regular modulo } I_H\}.\]
Suppose first that $I \cap S = \emptyset$. By \cite[Proposition 2.1.16(vii)]{MCR}, $I_S := I\cdot \Lambda_{G,H}$ is a prime ideal of $\Lambda_{G,H}$ and $I = I_S \cap \Lambda_G$. Because $I_S$ is reflexive by \cite[Proposition 1.2(a)]{AWZ}, it follows from Proposition \ref{Appl} and Theorem \ref{Omega} that $I_S = p\Lambda_{G,H}$. Hence $p \in I_S \cap \Lambda_G = I$, and therefore $I = p\Lambda_G$ by Lemma \ref{ModX}.

On the other hand, if $I \cap S \neq \emptyset$, then the right $\Lambda_G$-module $\Lambda_G / I$ is $S$-torsion and is therefore finitely generated over $\Lambda_H$ by \cite[Propositions 2.3 and 2.6]{CFKSV}.
\end{proof}

\subsection{Centre of $Q(\Lambda_H)$}\label{CentQ}

To obtain our second application of Theorem \ref{Omega}, we combine it with Proposition \ref{ModX}:

\begin{prop} Let $I$ be a prime c-ideal of $\Lambda_H$. Then $I = p\Lambda_H$.
\end{prop}

We can now compute the centre of the skewfield of fractions $Q$ of $\Lambda_H$.

\begin{thm} The centre of $Q$ is equal to $\Qp$.
\end{thm}
\begin{proof} Recall that $\Lambda_H$ is a maximal order in $Q$ by Lemma \ref{IwaMaxOrd}. Let $q$ be a non-zero central element of $Q$. Then $q\Lambda_H$ is a fractional c-ideal and hence can be written as a product of prime c-ideals of $\Lambda_H$ (and their inverses) by Theorem \ref{Asano}. Using the proposition, we see that any such product must be equal to $p^n\Lambda_H$ for some $n\in\mathbb{Z}$. Hence $q = p^nr$ for some $r\in\Lambda_H$ and this $r$ must clearly be a central element of $\Lambda_H$. Our assumptions on $H$ force the centre of $H$ to be trivial, so $Z(\Lambda_H) = \Zp$ by \cite[Corollary A]{A}. Hence $q \in \Qp$.
\end{proof}

Fix a topological generator $g \in Z$ such that $Z = \overline{\langle g \rangle}$ and write $z = g - 1\in \Lambda_Z$.
Motivated in part by Theorem \ref{FgOverH}, we now move on to study the prime ideals in $\Lambda_G \cong \Lambda_H[[z]]$ that are finitely generated over $\Lambda_H$.

\subsection{Proposition}\label{ZorH} Let $I$ be a prime ideal in $\Lambda_G$ such that $\Lambda_G/I$ is finitely generated over $\Lambda_H$. Then either $I\cap\Lambda_H \neq 0$ or $I \cap \Lambda_Z \neq 0$.
\begin{proof}
Let $A_n := \bigoplus_{i=0}^n\Lambda_H z^i$ --- this is a finitely generated $\Lambda_H$-submodule of $\Lambda_G$ and we have an increasing chain
\[ \Lambda_H = A_0 \subset A_1 \subset A_2 \subset \cdots .\]
The image of this chain inside $\Lambda_G / I$ must stabilize since $\Lambda_G / I$ is a Noetherian $\Lambda_H$-module by assumption, and it follows that $I \cap A_n \neq 0$ for some $n$.

Let $n$ be minimal subject to having $I \cap A_n \neq 0$: note that if $n=0$ then $I \cap \Lambda_H \neq 0 $ and we're done. So assume that $n \neq 0$; we can then find some non-zero polynomial
\[a = a_nz^n + a_{n-1}z^{n-1} + \cdots + a_1z + a_0 \in I.\]
Note that $a_n \neq 0$ by the minimality of $n$. Writing $Q$ for the skewfield of fractions of $\Lambda_H$, we can form the polynomial ring $Q[z]$. Note that $Q(I\cap \Lambda_H[z])$ is a two-sided ideal of $Q[z]$ by \cite[Proposition 2.1.16]{MCR}, and
\[a_n^{-1}a = z^n + (a_n^{-1}a_{n-1})z^{n-1} + \cdots + (a_n^{-1}a_1)z + a_n^{-1}a_0 \in Q(I\cap \Lambda_H[z]).\]
Let $u\in Q$ and consider the commutator $[u,a_n^{-1}a] = ua_n^{-1}a - a_n^{-1}au$. This lies in $Q(I\cap \Lambda_H[z])$ and has degree strictly smaller than $n$. Clearing denominators we obtain an element of $I \cap A_{n-1}$ which is zero by the minimality of $n$. Therefore $[u, a_n^{-1}a] = 0$ for any $u \in Q$ meaning that each $a_n^{-1}a_i$ is central in $Q$.  By Theorem \ref{CentQ}, $a_n^{-1}a_i\in\Qp$ for all $i < n$, so we can find some $m\geq 0$ such that $p^ma_n^{-1}a\in \Lambda_Z$.

Now $p^ma = a_n \cdot (p^ma_n^{-1}a) \in I$ and $I$ is prime. Because $\Lambda_Z$ is central in $\Lambda_G$ and $a_n \notin I$ (otherwise $I\cap \Lambda_H\neq 0$), we see that $I\cap \Lambda_Z \neq 0$.
\end{proof}

In the case when the ideal $I$ is actually a prime c-ideal, we will be able to refine the above. First, some preparatory results.

\subsection{Lemma}\label{GoingDown} Let $I$ be a prime ideal in $\Lambda_G$. Then
\begin{enumerate}[{(}a{)}]
\item $I\cap\Lambda_Z$ is a prime ideal in $\Lambda_Z$,
\item $J := I\cap\Lambda_H$ is a prime ideal in $\Lambda_H$,
\end{enumerate}
\begin{proof}
Let us remark that in the general setting of noncommutative rings, the inverse image of a prime ideal under a ring homomorphism need not be prime (or even semiprime!). This is clearly illustrated by the inclusion
\[\begin{pmatrix} k & k \\ 0 & k\end{pmatrix} \subset M_2(k)\]
for any field $k$. However in our setting there is enough commutativity to enable the obvious proofs of parts (a) and (b) to work. For part (b) let $A,B$ be ideals of $\Lambda_H$ such that $AB \subseteq J$; then $B\Lambda_G = B[[z]] = \Lambda_GB$ so $(A\Lambda_G)(B\Lambda_G) \subseteq J\Lambda_G \subseteq I$ and either $A\Lambda_G \subseteq I$ or $B\Lambda_G \subseteq I$ because $I$ is prime. Hence $A\subseteq J$ or $B\subseteq J$ whence $J$ is prime. The proof of part (a) is similar but easier.
\end{proof}

\subsection{Distinguished polynomials}\label{GoingUp}
Recall \cite[\S 7.1]{Wash} that a polynomial
\[f = a_n z^n + a_{n-1}z^{n-1} + \cdots + a_1z + a_0 \in \Zp[z]\]
is said to be \emph{distinguished} if $a_n = 1$ and $a_i \in p\Zp$ for all $i<n$.

\begin{lem}
Let $J\neq 0$ be an ideal in $\Lambda_Z$ such that $\Lambda_Z / J$ is $p$-torsionfree. Then
\begin{enumerate}[{(}a{)}]
\item $J$ is generated by a distinguished polynomial $f$,
\item if $J$ is a prime ideal of $\Lambda_Z$, then $J\Lambda_G = f\Lambda_G$ is a prime ideal of $\Lambda_G$.
\end{enumerate}
\end{lem}
\begin{proof} (a) This is a well-known consequence of the $p$-adic Weierstrass preparation theorem \cite[Theorem 7.3]{Wash}.

(b) Let $K$ denote the field of fractions of the domain $\Lambda_Z/J$. $K$ is a finite field extension of $\Qp$ because $\Lambda_Z/J$ is finitely generated over $\Zp$. By part (a), the image of $z$ in $K$ satisfies a monic polynomial with coeffcients in $\Zp$, so $\Lambda_Z/J$ is contained in the ring of integers $\mathcal{O}$ of $K$.

Now $\Lambda_G$ is isomorphic to the completed group ring of $H$ with coefficients in $\Lambda_Z$:
\[\Lambda_G \cong \lim\limits_{\longleftarrow}\!_{N\triangleleft_oH}\Lambda_Z[H/N]\]
and for every open normal subgroup $N$ of $H$ we have an exact sequence
\[0 \to \Lambda_Z[H/N] \stackrel{f}{\to} \Lambda_Z[H/N] \to \mathcal{O}[H/N].\]
Since inverse limits are left exact, we obtain an exact sequence
\[0 \to \Lambda_G \stackrel{f}{\to} \Lambda_G \to \mathcal{O}[[H]]\]
where $\mathcal{O}[[H]]$ denotes the completed group ring of $H$ with coefficients in $\mathcal{O}$. The latter is a domain by \cite[Proposition 2.5(iii)]{A2} so $\Lambda_G/f\Lambda_G$ is also a domain and $f\Lambda_G$ is a prime ideal.
\end{proof}

We are now ready to give proofs of our main results.

\subsection{Theorem}\label{PrimesG} Let $I$ be a prime c-ideal of $\Lambda_G$. Then $I = f\Lambda_G$ where $f = p$ or $f \in \Zp[z]$ is a distinguished polynomial which is irreducible in $\Lambda_Z$.
\begin{proof} Suppose that $J := I \cap \Lambda_H$ is non-zero. $J$ is prime by Lemma \ref{GoingDown}(b) and reflexive by \cite[Proposition 1.2(b)]{AWZ}, so $J = p\Lambda_H$ by Proposition \ref{CentQ}. Hence $p\Lambda_G \subseteq I$. Because $\Lambda_G /p\Lambda_G = \Omega_G$ is a domain by \cite[Theorem 4.3]{AB2}, $I = p\Lambda_G$ by Lemma \ref{ModX}.

Now suppose that $J = 0$; in particular, $p \notin I$. By Theorem \ref{FgOverH}, $\Lambda_G/I$ is finitely generated over $\Lambda_H$, so Proposition \ref{ZorH} implies that $I\cap \Lambda_Z \neq 0$. Because $p\notin I$, Lemma \ref{GoingUp}(a) implies that $I\cap\Lambda_Z = f\Lambda_Z$ for some distinguished polynomial $f\in\Zp[z]$. The polynomial $f$ must be irreducible in $\Lambda_Z$ by Lemma \ref{GoingDown}(a). Now $f\Lambda_G$ is a prime ideal in $\Lambda_G$ by Lemma \ref{GoingUp}(b), so $I = f\Lambda_G$ by Lemma \ref{ModX}.
\end{proof}

\begin{cor} Let $I$ be a reflexive two-sided ideal of $\Lambda_G$. Then $I = f\Lambda_G$ for some $f \in \Zp[z]$.
\end{cor}
\begin{proof}
Apply Proposition \ref{IwaMaxOrd}, Theorem \ref{Asano} and Theorem \ref{PrimesG}.
\end{proof}

\subsection{Proof of Theorem \ref{Main}}\label{Proof} Recall from Lemma \ref{IwaMaxOrd}
that $\Lambda_G$ is a maximal order. By \cite[Proposition
4.2.2]{Ch}, $q(M)$ decomposes uniquely in the quotient category
$\mathcal{M}/\mathcal{C}$ as a direct sum $q(M) = F \oplus B$ of a
completely faithful object $F$ and a \emph{locally bounded} object
$B$. Let $N$ be the $\Lambda_Z$-torsion submodule of $M$; we will
first show that $B = q(N)$.

Since $B$ is a subobject of $q(M)$, we can find a submodule $X$ of
$M$ such that $q(X) \cong B$ by general properties of quotient
categories. It will be convenient to identify $q(X)$ with $B$. Now,
since $X_o \leq M_o = 0$, the annihilator ideal $I := \Ann(B)$
equals $\Ann(X)$ by \cite[Lemma 2.5]{Rob}. Because $X$ is locally
bounded, $I \neq 0$ and therefore $X$ is a finitely generated,
torsion, bounded $\Lambda_G$-module. Therefore $I = \Ann(q(X))$ is a
non-zero reflexive two-sided ideal of $\Lambda_G$ by \cite[Lemma
5.3(i)]{CSS}, and hence $I = f\Lambda_G$ for some non-zero $f \in
\Lambda_Z$ by Corollary \ref{PrimesG}. Therefore $X \leq N$ and $B =
q(X) \leq q(N)$. On the other hand, $N$ is bounded because
$\Lambda_Z$ is central in $\Lambda_G$, so $q(N) \leq B$.

Now $q(M)$ is completely faithful if and only if $B = 0$, and $B =
0$ if and only if $N = 0$ because $N_o \leq M_o = 0$. \qed

\end{document}